\documentclass{amsart}      
\usepackage{amssymb,amsmath,graphicx}
\newtoks\prt 
\numberwithin{equation}{section}
 
\newtheorem{thm}{Theorem}[section]
\newtheorem{question}[thm]{Question} 
\newtheorem{lemma}[thm]{Lemma} 
 
\newtheorem{cor}[thm]{Corollary} 
\newtheorem{example}[thm]{Example}

\theoremstyle{definition}

\def\eqn#1$$#2$${\begin{equation}\label#1#2\end{equation}}

\def\diam{\operatorname{diam}} 
 
\def\ep{\varepsilon} 
 
\def\en{\mathbb N} 
\def\er{\mathbb R} 
 
\def\dist{\operatorname{dist}} 
 
\def\r{|}

\def\Ker{\operatorname{Ker}}

\def \reg {\partial _{\kern1pt\text{reg}}}

\def\la{\langle}
\def\ra{\rangle}

\def\dh{\widehat{\operatorname{d}}}
\def\clu#1#2{\operatorname{clust}_{#1^{**}}(#2)}

\def\wde#1{\widetilde{\delta}\left(#1\right)}
\def\de#1{\delta\left(#1\right)}

\newcommand{\ca}[2][]{\operatorname{ca}_{#1}\left(#2\right)}
\newcommand{\wca}[2][]{\widetilde{\operatorname{ca}}_{#1}\left(#2\right)}

\newcommand{\cc}[2][]{\operatorname{cc}_{#1}\left(#2\right)}
\newcommand{\wk}[2][X]{\operatorname{wk}_{#1}\left(#2\right)}

\begin{document}

\title{On quantitative Schur and Dunford-Pettis properties}
\author{Ond\v{r}ej F.K. Kalenda and Ji\v{r}\'{\i} Spurn\'y}

\address{Department of Mathematical Analysis \\
Faculty of Mathematics and Physic\\ Charles University\\
Sokolovsk\'{a} 83, 186 \ 75\\Praha 8, Czech Republic}

\email{kalenda@karlin.mff.cuni.cz}
\email{spurny@karlin.mff.cuni.cz}


\thanks{Our research was supported in part by the grant GA\v{C}R P201/12/0290. The second author was also
supported by The Foundation of Karel Jane\v{c}ek for Science and Research.}

\begin{abstract} 
We show that the dual to any subspace of $c_0(\Gamma)$ has the strongest possible quantitative version of the Schur property. Further, we establish relationship between the quantitative Schur property and quantitative versions of the Dunford-Pettis property. Finally, we apply these results to show, in particular, that any subspace of the space of compact operators on $\ell_p$ ($1<p<\infty$) with Dunford-Pettis property satisfies automatically both its quantitative versions.
\end{abstract}
\maketitle


\section{The main result}

A Banach space $X$ is said to have the \emph{Schur property} if any weakly null sequence in $X$ converges to zero in norm. Equivalently, $X$ has the Schur property if every weakly Cauchy sequence is norm Cauchy.
The classical example of a space with the Schur property is the space $\ell_1$ of all absolutely summable sequences.

A quantitative version of the Schur property was introduced and studied in \cite{qschur}. Let us recall the definition.
If $(x_k)$ is a bounded sequence in a Banach space $X$, we set (following \cite{qschur})
$$\ca{x_k}=\inf_{n\in\en}\diam\{x_k:k\ge n\}$$
and
$$\de{x_k}=\sup_{x^*\in B_{X^*}} \inf_{n\in\en} \diam\{x^*(x_k):k\ge n\}.$$
Then the quantity $\ca{\cdot}$ measures how far the sequence is from being norm Cauchy, while the quantity $\de{\cdot}$ measures how far it is from being weakly Cauchy. It is easy to check that the quantity $\de{x_k}$ can be alternatively described as the diameter of the set of all weak* cluster points of $(x_k)$ in $X^{**}$.
Following again \cite{qschur}, a Banach space $X$ is said to have the \emph{$C$-Schur property} (where $C\ge 0$) if
\begin{equation}
\label{eq:qsch1}
\ca{x_k}\le C \de{x_k}
\end{equation}
for any bounded sequence $(x_k)$ in $X$. Since obviously $\de{x_k}\le \ca{x_k}$ for any bounded sequence $(x_k)$, necessarily $C\ge 1$ (unless $X$ is the trivial space). Moreover, if $X$ has the $C$-Schur property for some $C\ge 1$, it easily follows that $X$ has the Schur property. Indeed, if $(x_k)$ is weakly Cauchy in $X$, then $\de{x_k}=0$, and thus $\ca{x_k}=0$. The space constructed in \cite[Example 1.4]{qschur} serves as an example of a Banach space with the Schur property without the $C$-Schur property for any $C>0$. On the other hand, $\ell_1(\Gamma)$ possesses the $1$-Schur property (see \cite[Theorem~1.3]{qschur}). 
Our main result is the following generalization of the quoted theorem.

\begin{thm}
\label{t:c01}
Let $X$ be a subspace of $c_0(\Gamma)$. Then $X^*$ has the $1$-Schur property.
\end{thm}

Let us now proceed to the proof of the main result.

We will need some lemmas. The first one establishes a special property of the norm on $c_0(\Gamma)$ and its subspaces.

\begin{lemma}\label{m1} Let $X$ be a subspace of $c_0(\Gamma)$. Then for any $x^*\in X^*$ and any sequence $(x_n^*)$ in $X^*$ which weak$^*$ converges to $0$ we have
$$\limsup\|x_n^*+x^*\|=\|x^*\|+\limsup\|x_n^*\|.$$
\end{lemma}

\begin{proof}
Let us first suppose that $X$ is separable. It is obvious that for any $x\in X$ and any weakly null sequence $(x_n)$ in $X$ we have
$$\limsup\|x_n+x\|=\max (\|x\|,\limsup\|x_n\|).$$
The assertion then follows from \cite[Theorem 2.6]{KaWe} (applied for $p=\infty$).

The general case follows by a separable reduction argument. Suppose that $x^*\in X^*$ and that $(x_n^*)$ is a weak* null sequence in $X^*$. Let us consider the countable set
$$A=\{x^*\}\cup\{x_n^*:n\in\en\}\cup\{x_n^*+x^*:n\in\en\}.$$
We can find a separable subspace $Y\subset X$ such that for each $y^*\in A$ we have $\|y^*\|=\|y^*|_Y\|$.
Then the assertion follows immediately from the separable case.
\end{proof}

 The next one is a stronger variant of \cite[Lemma~1.7]{brown} or \cite[Lemma 2.3]{KaWe} for the special case of subspaces of $c_0(\Gamma)$.

\begin{lemma}
\label{l:c01}
Let $X$ be a subspace of $c_0(\Gamma)$ and $(x_n^*)$ be sequence in $X^*$ weak$^*$ converging to $x^*$. Then for any finite dimensional subspace $F\subset X^*$ we have
\[
\liminf \dist(x_n^*, F)\ge\liminf\|x_n^*\|-\|x^*\|.
\]
\end{lemma}

\begin{proof}
Let $c>\liminf \dist(x_n^*, F)$ be arbitrary. By passing to a subsequence we may assume that $\dist(x_n^*, F)<c$
for each $n\in\en$. We can thus find a sequence $(y_n^*)$ in $F$ such that $\|x_n^*-y_n^*\|<c$ for each $n\in\en$.
Since the sequence $(x_n^*)$ is bounded, the sequence $(y_n^*)$ is bounded as well. Therefore we can, up to passing to a subsequence, suppose that the sequence $(y_n^*)$ converges in norm to some $y^*\in F$. 

Then

$$\begin{aligned}c&\ge\limsup\|x_n^*-y_n^*\|=\limsup\|x_n^*-y^*\|
=\limsup\|(x_n^*-x^*)+(x^*-y^*)\|
\\ &=\limsup\|x_n^*-x^*\|+\|x^*-y^*\|
\ge\limsup\|x_n^*\|-\|x^*\|+\|x^*-y^*\|
\\& \ge \liminf\|x_n^*\|-\|x^*\|.\end{aligned}$$

The first equality follows from the fact that the sequence $(y_n^*)$ converges to $y^*$ in the norm, the third
one follows from Lemma~\ref{m1}. The remaining steps are trivial. 

This completes the proof.
\end{proof}

The next lemma is a refinement of constructions from \cite[Lemma~2.1]{qschur} and \cite[Theorem~1.1]{brown}.
During its proof we will use the following notation: if $x\in c_0(\Gamma)$ or $x\in \ell_1(\Gamma)$ and $A\subset \Gamma$, then $x\r_A$ denotes an element defined as
\[
(x\r_A)(\gamma)=\begin{cases} x(\gamma), &\gamma\in A,\\
                            0,& \gamma\in \Gamma\setminus A.
             \end{cases}               
\]

\begin{lemma}\label{l:c02} Let $X$ be a subspace of $c_0(\Gamma)$, $c>0$ and $(y_n)$ be a sequence in $\ell_1(\Gamma)=c_0(\Gamma)^*$ such that
\begin{itemize}
	\item $(y_n)$ weak$^*$ converges to $0$ in $\ell_1(\Gamma)$,
	\item $\|y_n|_X\|>c$ for each $n\in\en$.
\end{itemize}
Then for any $\eta>0$ there is a subsequence $(y_{n_k})$ such that each weak$^*$ cluster point of $(y_{n_k}|_X)$ in $X^{***}$ has norm at least $c-\eta$.
\end{lemma}

\begin{proof} For $n\in\en$ set $\varphi_n=y_n|_X$.
Let $\ep\in (0,\frac{c}{6})$ be arbitrary. Without loss of generality, we may assume that $\ep<1$. 
We select strictly positive numbers $(\ep_k)$ such that $\sum_{k=1}^\infty \ep_k<\ep$.

We inductively construct elements $x_k\in X$, indices $n_1<n_2<\cdots$ and finite sets $\emptyset=\Gamma_0\subset\Gamma_1\subset\Gamma_2\subset\cdots\subset \Gamma$  such that, for each $k\in\en$,
\begin{enumerate}
\item [(a)] $\|x_k\|\le 1$, $x_k\r_{\Gamma_{k-1}}=0$ and $\|x_k\r_{\Gamma\setminus \Gamma_k}\|<\ep_k$,
\item [(b)] $|\varphi_{n_k}(x_k)|>c-\ep$ and $|\varphi_{n_k}(\sum_{i=1}^{k-1} x_i)|\le\ep\cdot \|\sum_{i=1}^{k-1} x_i\|$,
\item [(c)] if we denote $y_{n_k}^{1}=y_{n_k}\r_{\Gamma_k}$ and $y_{n_k}^{2}=y_{n_k}\r_{\Gamma\setminus \Gamma_k}$, then $\|y_{n_k}^{2}\|<\ep_k$.
\end{enumerate}

In the first step, we set $\Gamma_0=\emptyset$ and $n_1=1$. 
Since $\|\varphi_{n_1}\|>c$, there is $x_1\in B_X$ with $|\varphi_{n_1}(x_1)|>c$. 
Let us choose a finite set $\Gamma_1\subset \Gamma$ satisfying 
\[
\|x_1\r_{\Gamma\setminus \Gamma_1}\|<\ep_1\quad\text{and}\quad \|y_{n_1}\r_{\Gamma\setminus \Gamma_1}\|<\ep_1.
\]
Since the second requirement in (b) is vacuous, the first step is finished.

Assume now that we have found indices $n_1<\cdots <n_k$, finite sets $\emptyset=\Gamma_0\subset \cdots\subset \Gamma_{k}$ and elements $x_1,\dots x_k$ satisfying (a), (b) and (c).
We define an operator $R_k:X\to c_0(\Gamma)$ as
\[
R_kx=x\r_{\Gamma_k},\quad x\in X.
\]
Then $\Ker R_k$ is of finite codimension, and thus $F_k=(\Ker R_k)^\perp$ is a finite dimensional space in $X^*$.
Let $m\in\en$ be chosen such that, for each $n\ge m$,
\begin{itemize}
\item $|\varphi_n(\sum_{i=1}^{k-1} x_i)|\le\ep\cdot \|\sum_{i=1}^{k-1} x_i\|$, and
\item $\dist (\varphi_n, F_k)>c-\ep$.
\end{itemize}
(The first requirement can be fulfilled due to the fact that $(\varphi_n)$ converges weak$^*$ to $0$, and the second one due to Lemma~\ref{l:c01}.)
Let $n_{k+1}=m$ and 
\[
x_{k+1}\in (F_k)_\perp=\Ker R_k
\]
be chosen such that $\|x_{k+1}\|\le 1$ and 
\[
\varphi_{n_{k+1}}(x_{k+1})>c-\ep
\]
(we use the fact that $X^*/F_k=((F_k)_\perp)^*$).
We find a finite set $\Gamma_{k+1}\supset \Gamma_{k}$ satisfying 
\[
\|x_{k+1}\r_{\Gamma\setminus\Gamma_{k+1}}\|<\ep_{k+1}\quad\text{and}\quad\|y_{n_{k+1}}\r_{\Gamma\setminus\Gamma_{k+1}}\|<\ep_{k+1}.
\]
This finishes the construction.

For $J\in\en$, let 
\[
u_J=\sum_{i=1}^J x_i.
\]
It follows from (a) that, for each $k\in\en$ and $J>k$, we have
\begin{equation}
\label{e:c05.5}
\left\|\sum_{i=1}^k x_i\right\|<1+\ep,\quad 
\left\|\;\sum_{i=1}^{k-1} x_i\right\|< 1+\ep,\quad \left\|\;\sum_{i=k+1}^{J} x_i\right\|<1+\ep.
\end{equation}
Indeed, for $k\in\en$ and $\gamma\in \Gamma_k\setminus\Gamma_{k-1}$, we have from (a)
\[
|x_j(\gamma)|\le\begin{cases} \ep_j,& j<k,\\
1,& j=k,\\
0,& j>k,
\end{cases}
\quad j\in\en.
\]
Further, $x_k$ is bounded by $\ep_k$ on $\Gamma\setminus \bigcup_{k=1}^\infty\Gamma_k$ by (a). This observations verify \eqref{e:c05.5}.

For each $k\in\en$, we set
\[
\varphi_{n_k}^{1}=y_{n_k}^{1}\r_X\quad\text{and}\quad \varphi_{n_k}^{2}=y_{n_k}^{2}\r_X.
\]
For a fixed index $k\in\en$ and arbitrary $J>k$, we need to estimate
\begin{equation}
\label{e:c06}
|\varphi_{n_k}(u_J)|=\left|\varphi_{n_k}\left(\sum_{i=1}^{k-1} x_i\right)+\varphi_{n_k}(x_k)+\varphi_{n_k}\left(\sum_{i=k+1}^J x_i\right)\right|.
\end{equation}
The condition (b) and \eqref{e:c05.5} ensures that
\begin{equation}
\label{e:c08}
\left|\varphi_{n_k}\left(\sum_{i=1}^{k-1} x_i\right)\right|\le\ep\cdot\left\|\;\sum_{i=1}^{k-1} x_i\right\|<\ep (1+\ep).
\end{equation}

From (b) we also have
\begin{equation}
\label{e:c09}
\aligned
|\varphi_{n_k}(x_{k})|>c-	\ep.
\endaligned
\end{equation}
Finally, (a) and (c) give
\begin{equation}
\label{e:c010}
\aligned
\left|\;\varphi_{n_k}\left(\sum_{i=k+1}^J x_i\right)\right|&=\left|\left(\varphi_{n_k}^{1}+\varphi_{n_k}^{2}\right)\left(\sum_{i=k+1}^J x_i\right)\right|\\
&=\left|\;y_{n_k}^{2}\left(\sum_{i=k+1}^J x_i\right)\right|\le \ep_{k}\cdot\left\|\;\sum_{i=k+1}^J x_i\right\|\\
&<\ep_k(1+\ep).
\endaligned
\end{equation}

Using \eqref{e:c08}--\eqref{e:c010} in \eqref{e:c06}, we get
\begin{equation}
\label{e:c011}
\aligned
|\varphi_{n_k}(u_J)|&\ge c-\ep -\ep(1+\ep)-\ep_k(1+\ep)\\
&\ge c-\ep(3+2\ep)\ge c -5\ep.
\endaligned
\end{equation}
It follows from \eqref{e:c011} that, for $z_J=(1+\ep)^{-1}u_J$, we have $z_J\in B_X$ by \eqref{e:c05.5} and
\[
|\varphi_{n_k}(z_J)|>(1+\ep)^{-1}\left(c-5\ep\right),\quad k\in\en, J>k.
\]
Let $z^{**}\in B_{X^{**}}$ be a weak$^*$ cluster point of $(z_J)$. 
Then
\begin{equation}
\label{e:c012}
|\varphi_{n_k}(z^{**})|\ge (1+\ep)^{-1}\left(c-5\ep\right),\quad k\in\en.
\end{equation}
It follows that each weak$^*$ cluster point of $(\varphi_{n_k})$ has norm at least
$(1+\ep)^{-1}(c-5\ep)$.

This completes the proof, as given $\eta>0$, we can in the beginning choose $\ep$ such that
$$(1+\ep)^{-1}\left(c-5\ep\right)>c-\eta.$$
\end{proof}

Now we are ready to prove the theorem:

\begin{proof}[Proof of Theorem~\ref{t:c01}] 
Let $X$ be a subspace of $c_0(\Gamma)$ and $(x_n^*)$ be a sequence in $X^*$ bounded by a constant $M$. We consider arbitrary $0<c<\ca{x_n^*}$. We extract subsequences $(a_n)$ and $(b_n)$ from $(x_n^*)$ such that 
\begin{equation}
\label{e:c00}
c<\|a_n-b_n\|,\quad n\in\en.
\end{equation}
We denote $\varphi_n=a_n-b_n$, $n\in\en$. We extend $a_n$ to $A_n\in\ell_1(\Gamma)$ and
 $\varphi_n$ to  $z_n\in\ell_1(\Gamma)$ with preservation of the norm and set $B_n=A_n-z_n$. Then $B_n$ is an extension of $b_n$ (not necessarily preserving the norm). By passing to a subsequence if necessary, assume that $(A_n)$ converges pointwise (and hence weak$^*$ in $\ell_1(\Gamma)$) to some $A\in\ell_1(\Gamma)$ and $(B_n)$ converges pointwise to some $B\in \ell_1(\Gamma)$. (This is possible due to the fact that any sequence in $\ell_1(\Gamma)$ can be viewed as a sequence in $\ell_1(\Gamma')$ for a countable $\Gamma'\subset \Gamma$.) Then $(z_n)$ weak$^*$ converges to $A-B$.
 Set $y_n=z_n-A+B$ for $n\in\en$. Then $(y_n)$ weak$^*$ converges to $0$ and $\|y_n|_X\|>c-\|(A-B)|_X\|$ for each $n\in\en$. 
 
Let $\ep>0$ be arbitrary. By Lemma~\ref{l:c02}, there is a subsequence $(y_{n_k})$ such that 
each weak$^*$ cluster point of $(y_{n_k}|_X)$ in $X^{***}$ has norm at least
$$c-\|(A-B)|_X\|-\ep.$$

Let $a$ be a weak$^*$ cluster point of $(a_{n_k})$ in $X^{***}$. Let $(a_\tau)$ be a subnet of $(a_{n_k})$ weak$^*$ converging to $a$. Let $b$ be a weak$^*$ cluster point of the net $(b_\tau)$. Then $a$ and $b$ are weak$^*$ cluster points of $(x_n^*)$ in $X^{***}$.

Obviously $a|_X=A|_X$ and $b|_X=B|_X$ and, moreover, $a-b-(a-b)|_X=a-b-(A-B)|_X$ is a weak$^*$ cluster point of $(y_{n_k}|_X)$ in $X^{***}$. Thus
$$\|a-b-(a-b)|_X\|\ge c-\|(A-B)|_X\|-\ep.$$

Further, let $F\in (\ell_\infty(\Gamma))^*=c_0(\Gamma)^{***}$ be an extension of $a-b$ with preserving the norm. Then
$$\begin{aligned}\|a-b\|&=\|F\|=\|F|_{c_0(\Gamma)}\|+\|F-F|_{c_0(\Gamma)}\|
\ge \|F|_X\|+\|(F-F|_{c_0(\Gamma)})\r_{X^{**}}\|\\
&=\|(A-B)|_X\|+\|a-b-(a-b)|_X\|\\&\ge \|(A-B)|_X\|+c-\|(A-B)|_X\|-\ep\\
&=c-\ep.
\end{aligned}$$
(Let us remark that, for a Banach space $Y$ and $G\in Y^{***}$, we denote by $G|_Y$ the respective element of $Y^*$ canonically embedded into $Y^{***}$.) It follows that $\de{x_k^*}\ge c-\ep$. Since $\ep>0$ is arbitrary,
$\de{x_k^*}\ge c$. Hence  $\ca{x_k^*}\le \de{x_k^*}$ and the proof is completed.
\end{proof}

\section{Quantitative Schur property and quantitative Dunford-Pettis property}

It is well known that the Schur property is closely related to the Dunford-Pettis property.
Recall that a Banach space $X$ is said to have the \emph{Dunford-Pettis property} if for any Banach space $Y$ every 
weakly compact operator $T:X\to Y$ is completely continuous. Let us further recall that $T$ is \emph{weakly compact} if the image by $T$ of the unit ball of $X$ is relatively weakly compact in $Y$, and that $T$ is
\emph{completely continuous} if it maps weakly convergent sequences to norm convergent ones, or, equivalently, if it maps weakly Cauchy sequence to norm Cauchy (hence norm convergent) ones.

Obviously, any Banach space with the Schur property has the Dunford-Pettis property. Further,
any Banach space whose dual has the Schur property enjoys the Dunford-Pettis property as well.

Quantitative variants of the Dunford-Pettis property were studied in \cite{kks-adv} where two strengthenings of the Dunford-Pettis property in a quantitative way were introduced (\emph{direct quantitative Dunford-Pettis property} and \emph{dual quantitative Dunford-Pettis property}, see \cite[Definition~5.6]{kks-adv}).
Section 6 of \cite{kks-adv} shows several relations between the Schur property and the two variants ot the quantitative Dunford-Pettis properties. In this section we focus on the relationship of the quantitative Schur property and quantitative versions of the Dunford-Pettis property.

The unexplained notation and notions in this section are taken from \cite{kks-adv}. 

More specifically, the quantities $\ca[\rho^*]{\cdot}$ and $\ca[\rho]{\cdot}$ measure how far the given sequence is from being Cauchy in the Mackey topology of $X^*$ or the restriction to $X$ of the Mackey topology of $X^{**}$, respectively. The quantity $\wde{\cdot}$ is defined by taking infimum of $\de{\cdot}$ over all subsequences. Similarly for $\wca{\cdot}$, $\wca[\rho^*]{\cdot}$ and $\wca[\rho]{\cdot}$. These quantities are defined and described in detail in \cite[Section 2.3]{kks-adv}.

Further, $\dh(\cdot,\cdot)$ is the non-symmetrized Hausdorff distance, $\chi(\cdot)$ denotes the Hausdorff measure of norm non-compactness, $\omega(\cdot)$ and $\wk{\cdot}$ are measures of weak non-compactness; see \cite[Section 2.5]{kks-adv}. 	
To apply a measures of (weak) non-compactness to an operator means to apply it to the image of the unit ball (see \cite[Section 2.6]{kks-adv}).

Finally, the quantity $\cc{\cdot}$ measures how far the given operator is from being completely continuous, i.e. if $T:X\to Y$ is an operator, then
$$\cc{T}=\sup\{\ca{Tx_k}: (x_k)\mbox{ is a weakly Cauchy sequence in }B_X\},$$
see \cite[Section 2.4]{kks-adv}.

It is obvious that a Banach space $X$ with the Schur property possesses also the direct quantitative Dunford-Pettis property (see \cite[Proposition 6.2]{kks-adv}). If we assume that $X$ has a $C$-Schur property, we get the following result.

\begin{thm}
\label{p:qsch1}
Let $X$ be a Banach space with the $C$-Schur property where $C>0$.
\begin{itemize}
	\item[(i)] It holds $\ca[\rho]{x_n}\le C\de{x_n}$ for any bounded sequence $(x_n)$ in $X$. In particular, $X$ has both the direct and the dual quantitative Dunford-Pettis properties.  
	\item[(ii)] The space $X$ satisfies the following stronger version of the dual quantitative Dunford-Pettis property: If $A\subset X$ is a bounded set, then
\begin{equation}
\label{eq:qsch1.1}
\wk{A}\le \omega(A)=\chi(A)\le 2C\wk{A}.
\end{equation}                     
\end{itemize} 
\end{thm}

\begin{proof}
The inequality in assertion (i) follows from the fact that $\ca[\rho]{x_n}\le\ca{x_n}$ for any bounded sequence $(x_n)$ in $X$ (this is an immediate consequence of definitions). Thus $X$ satisfies condition (iv) of \cite[Theorem 5.5]{kks-adv}, i.e., $X$ possesses the dual quantitative Dunford-Pettis property. Further, from \cite[Proposition~6.2]{kks-adv} we know that $X$ has the direct quantitative Dunford-Pettis property. 

(ii) First we notice that \eqref{eq:qsch1.1} is indeed a stronger version of the dual quantitative Dunford-Pettis property. Indeed, using \cite[diagramm (3.1) and formula (2.6)]{kks-adv} one can deduce from \eqref{eq:qsch1.1} the validity of condition (i) of \cite[Theorem 5.5]{kks-adv}. 
 
For the proof of \eqref{eq:qsch1.1}, let $A$ be a bounded set in $X$. If $(x_k)$ in $X$ is a bounded sequence, by taking consecutively infima in \eqref{eq:qsch1} over all subsequences we obtain
\begin{equation}
\label{eq:qsch2}
\wca{x_k}\le C \wde{x_k}.
\end{equation}
By \cite[Theorem~1]{wesecom}, 
\begin{equation}
\label{eq:wesecom}
\wde{x_k}\le 2\dh (\clu{X}{x_k},X)
\end{equation}
for any bounded sequence $(x_k)$ in an arbitrary Banach space, 
and thus \eqref{eq:wesecom} together with \eqref{eq:qsch2} yield
\begin{equation}\label{eq:wca}
\wca{x_k}\le 2C\dh (\clu{X}{x_k},X).
\end{equation}
Since obviously (cf. \cite[inequalities (2.2)]{kks-adv})
\[
\chi(A)\le \sup\{\wca{x_k}: (x_k)\text{ is a sequence in }A\},
\]
\eqref{eq:wca} yields
\begin{equation}\label{eq:wca1}
\chi(A)\le 2C\wk{A}.
\end{equation}

Since $X$ has the $C$-Schur property, it has the Schur property, and thus any weakly compact subset of $X$ is norm compact. 
Hence
\begin{equation}
\label{eq:qsch5}
\chi(A)=\omega(A).
\end{equation}
A consecutive use of \cite[inequality (2.4)]{kks-adv}, \eqref{eq:qsch5},  and \eqref{eq:wca1} gives
\[
\wk{A}\le \omega(A)=\chi(A)\le 2C\wk{A},
\]
which is the inequality \eqref{eq:qsch1.1}.
\end{proof}

If the dual $X^*$ of a Banach space $X$ possesses the Schur property, then we have by \cite[Theorem~6.3]{kks-adv} that $X$ has the dual quantitative Dunford-Pettis property and, moreover, for any Banach space $Y$ and an operator $T:X\to Y$ the following inequalities hold:
\begin{equation}
\label{eq:qsch7}
\wk[Y]{T}\le \omega(T)\le\chi(T)\le\cc{T}\le 2\omega(T^*)=2\chi(T^*)\le 4\chi(T).
\end{equation}
Thus the quantities $\chi(T)$, $\cc{T}$, $\chi(T^*)$ and $\omega(T^*)$ are equivalent in this case. However, the quantities $\omega(T)$ and $\wk[Y]{T}$ need not be in this case equivalent with the others, i.e., $X$ need not have the direct quantitative Dunford-Pettis property, see \cite[Example~10.1]{kks-adv}.  However, if we assume that $X^*$ has a quantitative version of the Schur property, we obtain that, for an operator $T$ with domain $X$,  that the compactness (both norm and weak) of $T$ and its adjoint are quantitatively equivalent to the complete continuity of $T$.

\begin{thm}
\label{t:qsch1} Let $X$ be a Banach space such that $X^*$ have the $C$-Schur property for some $C\ge 0$.  If $Y$ is a Banach space and $T:X\to Y$ is a bounded linear operator, we have

\begin{equation}
\label{eq:qsch6}
\begin{aligned}
\wk[Y]{T}\le \omega(T)&\le\chi(T)\le\cc{T}\\ &\le 2\omega(T^*)=2\chi(T^*)\le 4C\wk[X^*]{T^*}
\le8C\wk[Y]{T}.
\end{aligned}
\end{equation}
In particular, $X$ has both the direct and the dual quantitative Dunford-Pettis properties.
\end{thm}

\begin{proof}
 The first five inequalities are contained in \cite[Theorem 6.3(i)]{kks-adv}. By Theorem~\ref{p:qsch1} we get the sixth inequality. The last inequality follows from \cite[equation (2.8)]{kks-adv}.
 Further, $X^*$ has both the direct and dual quantitative Dunford-Pettis property by Theorem~\ref{p:qsch1}(i). Hence $X$ itself possesses both the direct and dual quantitative Dunford-Pettis property by \cite[Theorem~5.7]{kks-adv}.
\end{proof}

If we combine the previous theorem with Theorem~\ref{t:c01}, we get immeadiately.

\begin{cor} Let $X$ be a subspace of $c_0(\Gamma)$. Then $X$ has both the direct and dual quantitative Dunford-Pettis properties. Moreover, the inequalities \eqref{eq:qsch6} are satisfied with $C=1$.
\end{cor}

In case $X=c_0(\Gamma)$ Theorem 8.2 of \cite{kks-adv} yields even stronger inequalities (with $C=1/2$).
The proof of this case is done by a different method.

\smallskip

We continue by a characterization of spaces whose dual has the quantitative Schur property.
It is well known that the dual space $X^*$ of a Banach space $X$ has the Schur property if and only if $X$ has the Dunford-Pettis property and contains no copy of $\ell_1$ (see \cite[Theorem~3]{diestel}). The following theorem quantifies this assertion.

\begin{thm}
\label{t:qsch-qdp}
Let $X$ be a Banach space. Then $X^*$ has the quantitative Schur property if and only if $X$ has the direct quantitative Dunford-Pettis property and contains no copy of $\ell_1$.
\end{thm}

\begin{proof}
Suppose that $X^*$ has the quantitative Schur property. Then $X$ contains no copy of $\ell_1$. Indeed, if $X$ contains an isomorphic copy of $\ell_1$, by \cite[Proposition 3.3]{pelc} the dual space $X^*$ contains an isomorphic
copy of $C(\{0,1\}^\en)^*$, hence also an isomorphic copy if $C([0,1])^*$. The space $C([0,1])^*$ fails 
the Schur property as it contains a copy of $L^1(0,1)$. Thus $X^*$ fails the Schur property as well.
Further, $X$ has the direct quantitative Dunford-Pettis property by Theorem~\ref{p:qsch1}.

For the proof of the converse implication we need the following consequence of Rosenthal's $\ell_1$-theorem.

\begin{lemma}\label{lm-nonell1} Let $X$ be a Banach space not containing an isomorphic copy of $\ell_1$. 
Then any bounded sequence $(x_n^*)$ in $X^*$ satisfies $\ca{x_n^*}\le 3\ca[\rho^*]{x_n^*}$.
\end{lemma}

\begin{proof} 
If $(x_n^*)$ is norm-Cauchy, then the inequality is obvious. So, suppose that $\ca{x_n^*}>0$ and fix any $c\in(0,\ca{x_n^*})$. Then there is a sequence of natural numbers $l_n<m_n<l_{n+1},\,n\in\en,$ and a sequence $(x_n)$ in~$B_X$ such that $|(x_{l_n}^*-x_{m_n}^*)(x_n)|>c$ for every $n\in\en$. By Rosenthal's $\ell_1$-theorem, there is a weakly Cauchy subsequence of $(x_n)$. Let us assume, without loss of generality, that $l_n=2n-1$ and $m_n=2n$ for every $n\in\en$ and that $(x_n)$ is weakly Cauchy. 

Since, for every $k\in\en$, the singleton $\{x_k\}$ is a weakly compact set in~$B_X$, there is some $n_k>k$ such that $|(x_{2n_k-1}^*-x_{2n_k}^*)(x_k)|<\ca[\rho^*]{x_n^*}+\frac1k$. Using this estimate and the fact that $\{\frac{x_{n_k}-x_k}2:k\in\en\}$ is a relatively weakly compact subset of $B_X$, we can write
\begin{eqnarray*}
c &\le& \limsup|(x_{2n_k-1}^*-x_{2n_k}^*)(x_{n_k})|\\
&\le& 2\limsup|(x_{2n_k-1}^*-x_{2n_k}^*)(2^{-1}(x_{n_k}-x_k))|+\limsup|(x_{2n_k-1}^*-x_{2n_k}^*)(x_k)|\\
&\le&2\ca[\rho^*]{x_n^*}+\limsup(\ca[\rho^*]{x_n^*}+\tfrac1k)=3\ca[\rho^*]{x_n^*}.
\end{eqnarray*}
This completes the proof.
\end{proof} 

Suppose now that $X$ has the direct Dunford-Pettis property. Then there exists $C>0$ such that
\[
\ca[\rho^*]{x_n^*}\le C\de{x_n^*}
\]
for any bounded sequence $(x_n^*)$ in $X^*$ (see \cite[Theorem~5.4(iv)]{kks-adv}). By Lemma~\ref{lm-nonell1},
\[
\ca{x_n^*}\le 3\ca[\rho^*]{x_n^*}\le 3C\de{x_n^*}
\]
for any bounded sequence $(x_n^*)$ in $X^*$.
Hence $X^*$ has the $3C$-Schur property.
\end{proof}

\section{Subspaces of $C(K)$, $K$ scattered}

It is natural to ask whether Theorem~\ref{t:c01} holds for larger class of spaces in place of $c_0(\Gamma)$. The first attempt is to consider isomorphic $\ell_1$ preduals, i.e., spaces whose dual is isomorphic to $\ell_1$. But this has no chance due to the old result of Bourgain and Delbaen \cite{BoDa} later improved by Haydon \cite{haydon-BD}. In fact, Freeman, Odell and Schlumprecht recently proved in \cite{FOS} that any Banach space with separable dual can be embedded into a space whose dual is isomorphic to $\ell_1$.

The second attempt is to consider isometric $\ell_1$ preduals, i.e., spaces whose dual is isometric to $\ell_1$ (or, more generally, $\ell_1(\Gamma)$). We focus on the case $C(K)$, $K$ scattered. We can substitute $C(K)$ for $c_0(\Gamma)$ if and only if $K$ has finite Cantor-Bendixson rank. But, of course, the constant $1$ should be substituted by a constant depending on the height of $K$. The positive part of this result is contained in Theorem~\ref{t:vyskan} below which essentially follows from the Bessaga-Pe\l{}czy\'nski classification of $C(K)$, $K$ countable. In Example~\ref{exa:vn} we show that the constant really depends on the height. This example can be viewed as an approximation of the example constructed in \cite{PeSz} which is recalled as a part of Example~\ref{exa:ps} below.

\begin{thm}\label{t:vyskan} Denote for $n\in \en$ the Banach-Mazur distance of $c_0$ and $C[0,\omega^n]$ by $C_n$. Let  $K$ be a compact space  satisfying $K^{(n+1)}=\emptyset$ for some $n$ and let $X$ be a Banach space isometric to a subspace of $C(K)$. Then $X^*$ has the $C_{n+1}$-Schur property.
\end{thm}

\begin{proof} Let $n$, $K$ and $X$ satisfy the assumptions. Firstly, we will show that without loss of generality we may assume that $X$ is separable.

Indeed, let $(x_k^*)$ be any bounded sequence in $X^*$. Denote by $Z$ the closed linear span of this sequence. Then $Z$ is separable, let $D$ be a countable norm-dense subset of $Z$. It is now easy to find a separable subspace $Y\subset X$ such that
$\|x^*|_Y\|=\|x^*\|$ for each $x^*\in D$. Then the mapping $x^*\mapsto x^*|_Y$ is an isometric injection of $Z$ into $Y^*$.
Therefore the quantities $\ca{x_k^*}$ and $\de{x_k^*}$ are the same when computed in $X^*$, $Z$ or $Y^*$. Therefore, if we know that $Y^*$ has the $C_{n+1}$-Schur property, we deduce that  $\ca{x_k^*}\le C_{n+1}\de{x_k^*}$. Since $(x_k^*)$ was arbitrary, this proves the $C_{n+1}$-Schur property of $X^*$.

So, in the rest of the proof we will suppose that $X$ is separable. Let $\tilde X$ be the closed algebra generated by $X$ and constant function $1$. Then $\tilde X$ is canonically isometric to $C(L)$, where $L$ is a quotient of $K$. (This is a well-known consequence of the Stone-Weierstrass theorem: Define on $K$ an equivalence $\sim$ by $k\sim l$ if and only if
$x(k)=x(l)$ for $x\in \tilde X$ (equivalently for $x\in X$). Then $L=K\big/_\sim$ is a compact space and $\tilde X$ is isometric to $C(L)$.) Since $X$ is separable, $\tilde X$ is separable as well, hence $L$ is metrizable. Further, since $K^{(n+1)}=\emptyset$, we get also $L^{(n+1)}=\emptyset$. (Indeed, let $q$ be the quotient mapping of $K$ onto $L$.
It is easy to check that $L'\subset q(K')$ and by induction we get $L^{(k)}\subset q(K^{(k)})$ for $k\in\en$.)

Therefore, without loss of generality $K$ is countable. It follows that $K$ is homeomorphic to $[0,\alpha]$ for an ordinal $\alpha<\omega^{n+1}$. Since $C[0,\alpha]$ is isometric to a subspace of $C[0,\beta]$ for $\alpha<\beta$, $X$ is isometric to a subspace of $C[0,\omega^{n+1}]$. This space is isomorphic to $c_0$ by the Bessaga-Pe\l czy\'nski classification of $C(K)$, $K$ countable. Let $d>C_{n+1}$ be arbitrary. It follows that there is an onto isomorphism $T:C[0,\omega^{n+1}]\to c_0$ with $\|T\|\cdot\|T^{-1}\|<d$.  Then $T(X)$ is an isometric subspace of $c_0$, so $T(X)^*$ has the $1$-Schur property by Theorem~\ref{t:c01}. Further, $S=(T|_X)^*$ is an isomorphism of $X^*$ onto $T(X)^*$ with $\|S\|\cdot\|S^{-1}\|<d$.
Let $(x_n^*)$ be a bounded sequence in $X^*$. Then
$$\ca{x_n^*}\le \|S^{-1}\|\ca{S x_n^*} = \|S^{-1}\|\de{S x_n^*}\le\|S^{-1}\|\cdot\|S\|\de{x_n^*}\le d \de{x_n^*}.$$
Hence $X^*$ has the $d$-Schur property. Since $d>C_{n+1}$ was arbitrary, $X^*$ has the $C_{n+1}$-Schur property.
\end{proof}

\begin{thm}\label{exa:vn} For each $n\in\en$ there exists a Banach space $X_n$ with the following properties.
\begin{itemize}
	\item[(i)] $X_n$ is isomorphic to $c_0$.
	\item[(ii)] $X_n$ is isometric to a subspace of $C[0,\omega^{n+1}]$.
	\item[(iii)] There are sequences $(e_k)$ in $X_n$ and $(e_k^*)$ in $X_n^*$ with the following properties:
	\begin{itemize}
	  \item[(a)] $\|e_k\|_n=1$ for each $k\in\en$.
	  \item[(b)] The sequence $(e_k)$ converges weakly to zero in $X_n$.
	  \item[(c)] For any $x^{**}\in X_n^{**}$ with $\|x^{**}\|\le 1$ we have $\limsup |x^{**}(e_k^*)|\le \frac 2n$.
	  \item[(d)] $x_k^*(x_k)=1$ for each $k\in \en$.
  \end{itemize}
  In particular, $X_n$ does not have the direct quantitative Dunford-Pettis property in the sense of \cite[Theorem 5.4(iii)]{kks-adv} with constant $C<\frac n2$ and $X_n^*$ does not have the $C$-Schur property for $C<\frac n{16}$.
\end{itemize}
\end{thm}

\begin{proof} For $x=(x_k)\in c_0$ and $p\in\en$ set
$$A_p(x) = \sup \{ |x_p+x_{i_1}+\dots+x_{i_p}| : p<i_1<i_2<\dots< i_p \}.$$
It is easy to check that
$$\max\left(|x_p|,\frac12\max\{|x_k|:k> p\}\right)\le A_p(x)\le (p+1)\|x\|_\infty$$
for any $x\in c_0$. 
Next let us fix some $n\in\en$. For $x\in c_0$ we set
$$\|x\|_n=\max\{A_p(x) : 1\le p\le n\}.$$
It follows that $\|\cdot\|_n$ is an equivalent norm on $c_0$. Set $X_n=(c_0,\|\cdot\|_n)$.
Then (i) is obviously fulfilled.

Let us show (ii). For $k\in\en$ let $\Lambda_k$ be the subset of $[0,\omega]^k$ formed by non-decreasing sequences equipped with the lexicographic order. Then it is easy to check that $\Lambda_k$ is order-isomorphic to the ordinal interval $[0,\omega^k]$. Further, set $\Lambda=\bigcup_{k=1}^{n+1}\Lambda_k$. Let us define an order on $\Lambda$ such that the shorter sequences are smaller and $\Lambda_k$ is ordered lexicographically. The set $\Lambda$ equipped with this order is order-isomorphic to the ordinal interval $[0,\omega^{n+1}]$.
Let us define a mapping $\varphi:X_n\to \er^\Lambda$ by the formula
$$\varphi(x)(i_1,\dots,i_k)=\begin{cases}\begin{array}{r} x(i_1)+\dots+x(i_l)\\ \quad \end{array} & 
\begin{array}{l} \mbox{if } 1\le i_1\le n \mbox{ and $l\le k$ is maximal} \\
\mbox{such that } l-1\le i_1<i_2<\dots<i_l,\end{array}  \\ 0 & \mbox{ otherwise},\end{cases}$$
where we use the convention that $x(\omega)=0$.
Then $\varphi$ is a well-defined isometry of $X_n$ into $\ell_\infty(\Lambda)$ and, moreover, $\varphi(X_n)\subset C(\Lambda)$ (where $\Lambda$ is considered with the order topology). 

Indeed, the inequality $\|x\|_n\le \|\varphi(x)\|_\infty$ is obvious.  To see the converse one  we fix $l$ and $i_1,\dots,i_l$ such that $1\le i_1\le n$ and $l-1\le i_1<i_2<\dots<i_l$. First suppose that $i_l<\omega$. If $l-1=i_1$, then
$$|x(i_1)+\dots+x(i_l)|\le A_{l-1}(x)\le \|x\|_n.$$
If $l-1<i_1$, then
$$|x(i_1)+\dots+x(i_l)|=\lim_{m\to\infty}|x(i_1)+\dots+x(i_l)+\sum_{j=1}^{i_1-l+1} x(i_l+k+j)|\le A_{i_1}(x)\le\|x\|_n.$$
If $i_l=\omega$, then $x(i_l)=0$, hence
$$|x(i_1)+\dots+x(i_l)|=|x(i_1)+\dots+x(i_{l-1})|\le\|x\|_n$$
by the previous case. This completes the proof that $\varphi$ is an isometry.

Finally, let us show that $\varphi(x)$ is a continuous function for each $x\in X_n$. Fix $x\in X_n$ and $\boldsymbol i=(i_1,\dots,i_k)\in\Lambda$. If $i_k<\omega$, then  $\boldsymbol i$ is an isolated point of $\Lambda$, so $\varphi(x)$ is continuous at this point. So, suppose that $i_k=\omega$. If $i_1>n$, then
$$\{(j_1,\dots,j_k): j_1>n\}=\{\boldsymbol j\in\Lambda_k: \boldsymbol j> (n,\omega,\dots,\omega)\}$$ is an open set containing $\boldsymbol i$ on which $\varphi(x)$ is zero. Hence $\varphi(x)$ is continuous at $\boldsymbol i$. Similarly, if $i_1=0$, then
$$\{(j_1,\dots,j_k): j_1=0\}=\{\boldsymbol j\in\Lambda_k: \boldsymbol j< (1,1,\dots,1)\}$$ is an open set containing $\boldsymbol i$ on which $\varphi(x)$ is zero. Next suppose that $1\le i_1\le n$. Let $m$ be the smallest index such that $i_m=\omega$. Necessarily $m\ge 2$. For $r\in\en$, $r\le i_{m-1}$, let
$$V_r=\{(i_1,\dots,i_{m-1},j_m,j_{m+1},\dots,j_k): j_m>r\} = \big((i_1,\dots,i_{m-1},r,\omega,\dots,\omega),\boldsymbol i\big].$$
These sets form a neighborhood basis of $\boldsymbol i$ in $\Lambda$. Let $l$ be maximal such that
$l-1\le i_1<i_2<\dots<i_l$. Necessarily $l\le m$. If $l<m-1$, then $\varphi(x)$ is constant on $V_{i_{m-1}}$.
If $l=m-1$, then necessarily $l-1=i_1$ and hence $\varphi(x)$ is again constant on $V_{i_{m-1}}$. If $l=m$, then
$\varphi(x)(\boldsymbol i)=x(i_1)+\dots+x(i_{m-1})$. For $\boldsymbol j\in V_r$ we have
$$|\varphi(x)(\boldsymbol i)-\varphi(x)(\boldsymbol j)|\le (k-l+1)\sup_{j\ge r}|x(j)|.$$
Since $x\in c_0$, this shows that $\varphi(x)$ is continuous at $\boldsymbol i$ and completes the proof that $\varphi(x)\in C(\Lambda)$.

Finally, let us prove (iii).
Let $(e_k)$ be the canonical basis of $X_n$. It follows from the definition that $\|e_k\|_n=1$ for each $k$, hence (a) holds.
Moreover, $e_k$ converges weakly to zero, as it is the case in $c_0$, so (b) holds as well.. Let $(e_k^*)$ be the sequence of biorthogonal functionals to $(e_k)$. Then clearly (d) holds.

Let us show (c). Fix any $x^{**}\in X_n^{**}$ satisfying $\|x^{**}\|\le 1$. Suppose for contradiction that $\limsup |x^{**}(e_k^*)|> \frac 2n$. Then there is $\eta>\frac2n$ such that $|x^{**}(e_k^*)|>\eta$ for infinitely many $k$. Without loss of generality we may suppose that $x^{**}(e_k^*)>\eta$ for infinitely many $k$ (otherwise we would replace $x^{**}$ by $-x^{**}$).
Therefore we can find indices $n<i_1<i_2<\dots<i_n$ such that $x^{**}(e_{i_j}^*)>\eta$ for $j=1,\dots,n$. By Goldstine theorem there is $x\in B_{X_n}$ with $e_{i_j}^*(x)>\eta$ for $j=1,\dots,n$. Then
$$1\ge\|x\|_n\ge A_n(x) \ge |x_n+x_{i_1}+\dots+x_{i_n}|\ge n\eta-|x_n|\ge n\eta-A_n(x)\ge n\eta-1,$$
hence $\eta\le\frac2n$, a contradiction.

It remains to prove the `in particular' part. It follows from (c) that any $w^*$-cluster point of $(e_k^*)$ in $X_n^{***}$ has norm at most $\frac2n$. Now it immeadiately follows that the space $X_n$ does not satisfy the condition (iii) of \cite[Theorem 5.4]{kks-adv} for $C<\frac n2$. Further, if $X_n^*$ has the $C$-Schur property, by Theorem~\ref{p:qsch1}(i) $X_n^*$ satisfies the condition (iv) of \cite[Theorem 5.6]{kks-adv} with constant $C$. It follows from the proof of the quoted theorem that then $X_n^*$ satisfies the condition (iii) of that theorem with constant $8C$. Finally, by \cite[Remark 5.8]{kks-adv} the space $X_n$ satisfies the condition (iii) of \cite[Theorem 5.4]{kks-adv} with the same constant $8C$. By the above we get $8C\le\frac n2$, hence $C\le\frac n{16}$ and the proof is completed.
\end{proof}

\begin{example}\label{exa:ps} There are Banach spaces $Y_1$ and $Y_2$ with the following properties:
\begin{itemize}
	\item[(i)] Both $Y_1$ and $Y_2$ are isometric to subspaces of $C[0,\omega^\omega]$.
	\item[(ii)] $Y_1$ fails the Dunford-Pettis property, so $Y_1^*$ fails the Schur property.
	\item[(iii)] $Y_2^*$ has the Schur property, so $Y_2$ has the dual quantitative Dunford-Pettis property.
	 $Y_2$ fails the direct quantitative Dunford-Pettis property 
\end{itemize}
\end{example}

\begin{proof} The existence of $Y_1$ is a result of \cite{PeSz}. The above spaces $X_n$ are in a sense approximations of $Y_1$. For the sake of completeness let us recall the definition of $Y_1$.

The quantity $A_p(x)$ defined above has sense for any $x\in\er^\en$. Furher, it is finite if and only if $x$ is bounded.
The space $(Y_1,\|\cdot\|)$ is defined as
$$Y_1=\{x\in \ell_\infty: A_p(x)\to 0\},\qquad \|x\|=\sup\{A_p(x):p\in\en\}.$$
It is proved in \cite{PeSz} that $Y_1$ is isometric to a subspace of $C[0,\omega^\omega]$, 
the canonical basis $(e_k)$ of $Y_1$ is unconditional, the orthogonal functionals $(e_k^*)$ form also an uncoditional basis of $Y_1^*$, $e_k$ weakly converge to zero, $e_k^*$ as well, while $e_k^*(e_k)=1$. This proves the failure of the Dunford-Pettis property.

The space $Y_2$ can be taken to be the $c_0$-sum of the spaces $X_n$, $n\in\en$. Then all the properties easily follow.
\end{proof}

In view of the previous example the following question seems to be natural.

\begin{question} Let $X$ be a Banach space isometric to a subspace of $C(K)$ with $K$ scattered. Suppose that $X$ has the Dunford-Pettis property. Does $X^*$ have the Schur property? Does $X$ have the dual quantitative Dunford-Pettis property?
\end{question}

A related topic is the study of spaces having hereditary Dunford-Pettis property, i.e., spaces all whose subspaces enjoy the Dunford-Pettis property. Within $C(K)$ spaces they are exactly those such that $K$ has finite height (as explicitely formulated in \cite[Theorem 1]{cembranos} as a consequence of \cite{PeSz}). Further, spaces with the Schur property enjoy hereditary Dunford-Pettis property as well. Further, the space constructed by Hagler in \cite{hagler} has also hereditary Dunford-Pettis property by \cite[Proposition 2]{cembranos}. It seems to us that the following questions are interesting.

\begin{question} \ 
\begin{itemize}
	\item[(1)] Is the space $Y_2$ from Example~\ref{exa:ps} hereditarily Dunford-Pettis?
	\item[(2)] Let $X\subset C(K)$ with $K$ scattered be hereditarily Dunford-Pettis. Is $X$ contained in $C(L)$ for some $L$ with finite height?
	\item[(3)] Does $X^*$ have the Schur property for any subspace of the space from \cite{hagler}?
\end{itemize}
\end{question}

\section{Subspaces of the space of compact operators}

The space $K(\ell_2)$ of all compact operators on the Hilbert space $\ell_2$ can be viewed as a non-commutative version of $c_0$ and its dual $N(\ell_2)$, the space of all nuclear operators on $\ell_2$ equipped with the nuclear norm, can be viewed as a non-commutative version of $\ell_1$. The non-commutative versions share many properties of the commutative ones, but Schur property and Dunford-Pettis property are essentially commutative.

Indeed, $N(\ell_2)$ does not have the Schur property and, moreover, $K(\ell_2)$ does not enjoy the Dunford-Pettis property. It is witnessed by the following easy example. Let $(e_n)$ denote the standard basis in $\ell_2$. Consider the operators $T_n(x)=\la x,e_1\ra e_n$, $x\in\ell_2$, and $S_n(x)=\la x,e_n\ra e_1$. These operators are rank-one operators, thus they are nuclear and hence compact. Moreover, both sequences converge weakly to $0$ both in $K(\ell_2)$ and $N(\ell_2)$. The Schur property of $N(\ell_2)$ can be disproved by observing that $\|S_n\|=\|T_n\|=\|e_1\|\|e_n\|=1$. Moreover, the failure of the Dunford-Pettis property of $K(\ell_2)$ follows by
the fact that $\operatorname{Tr}(S_nT_n)=1$.

This easy observation was strengthened in \cite{sa-ty}, where the authors show that a subspace of $K(\ell_p)$,
the space of compact operators on $\ell_p$ enjoys the Dunford-Pettis property if and only if it is isomorphic
to a subspace of $c_0$ (i.e., only in the ``commutative case''). Theorem~\ref{t:c01} enables us to complement and strengthen their result to show that such a space has automatically a quantitative Dunford-Pettis property.

More precisely, we prove the following:

\begin{thm}
\label{t:kop}
Let $X$ be a subspace of the space $K(\ell_p)$ of compact operators on $\ell_p$ where $1<p<\infty$. Then the following assertions are equivalent:
\begin{enumerate}
\item[(i)] $X$ has the Dunford-Pettis property.
\item[(ii)] $X^*$ has the Schur property.
\item[(iii)] $X$ is isomorphic to a subspace of $c_0$. Moreover, in this case, there is for each $\ep>0$ an  isomorphic embedding $T:X\to c_0$ such that $\|T\|\|T^{-1}\|<4+\varepsilon$.
\item[(iv)] $X^*$ has the $4$-Schur property.
\item[(v)] For each Banach space $Y$ and each bounded linear operator $T:X\to Y$, the inequalities \eqref{eq:qsch6} hold with $C=4$.
\item[(vi)] The space $X$ has both the dual and the direct quantitative Dunford-Pettis properties.
\end{enumerate}
\end{thm}

\begin{proof} The implication (ii) $\Rightarrow$ (i) is well known (see \cite[Theorem 3]{diestel}). 

(i) $\Rightarrow$ (iii)
If  $X\subset K(\ell_p)$ has the Dunford-Pettis property, it is embeddable into $c_0$ by \cite[Theorem~1]{sa-ty}. Moreover, the constant of embedding can be explicitly computed from \cite[Lemma~1 and~2]{sa-ty}. Indeed, the embedding $T:X\to c_0$ is constructed as the composition $\psi\circ\phi_A$, where $\phi_A$ is provided by \cite[Lemma~1]{sa-ty} and $\psi$ is provided by \cite[Lemma~2]{sa-ty}.
The operator $\psi$ satisfies $\|\psi\|\|\psi^{-1}\|\le 4$ by \cite[p. 420]{sa-ty}.
Further, $\phi_A$ satisfies $\|\phi_A\|\|\phi_A^{-1}\|\le 3$ (see the computation in \cite[p. 418]{sa-ty}), but it can be easily modified to be an almost isometry. Indeed, if we replace in \cite[formula (3) on p. 420]{sa-ty} the number $\frac14$ by $\frac\ep2$, then we will obtain $\|\phi_A\|\|\phi_A^{-1}\|\le\frac{1+\ep}{1-\ep}$. This completes the proof.
 
The implication (iii) $\Rightarrow$ (iv) follows from Theorem~\ref{t:c01}. Indeed, let $T:X\to c_0$ be an embedding with $\|T\|=1$ and $\|T^{-1}\|\le 4+\ep$. Let $(x^*_n)$ be a bounded sequence in $X^*$. Then $((T^*)^{-1}x_n^*)$ is a bounded sequence in $(T(X))^*$ satisfying $\de{(T^*)^{-1}x_n^*}\le(4+\ep)\de{x_n^*}$. 
By Theorem~\ref{t:c01} we get $\ca{(T^*)^{-1}x_n^*}\le(4+\ep)\de{x_n^*}$, hence
$\ca{x_n^*}\le(4+\ep)\de{x_n^*}$ as well. Since $\ep>0$ is arbitrary, the proof is finished.

The implications (iv) $\Rightarrow$ (v) and (v) $\Rightarrow$ (vi) follows from Theorem~\ref{t:qsch1}.
Finally, the implications (vi) $\Rightarrow$ (i) and (iv) $\Rightarrow$ (ii) are trivial.
\end{proof}

\def\cprime{$'$}

\end{document}